\documentclass[10pt,oneside,leqno]{amsart}
\usepackage{amsxtra}
\usepackage{amsopn}
\usepackage{amsmath,amsthm,amssymb}
\usepackage{amscd}
\usepackage{amsfonts}
\usepackage{latexsym}
\usepackage{verbatim}

\theoremstyle{plain}
\newtheorem{theorem}{Theorem}[section]
\newtheorem{definition}[theorem]{Definition}

\newtheorem{prop}[theorem]{Proposition}

\newtheorem{rem}[theorem]{Remark}

\newcommand\B{{\mathbb B}}
\newcommand\C{{\mathbb C}}
\newcommand\Iw{{\mathbb I}(3)}
\newcommand\GL{{\hbox{\rm GL}}}
\newcommand\CP{{\mathbb {CP}}}

\newcommand\R{{\mathbb R}}
\newcommand\Z{{\mathbb Z}}

\newcommand\T{{\mathbb T}} \newcommand\Proj{{\mathbb P}}
\begin{document}
\title[Blow-ups and resolutions of strong K\"ahler with torsion metrics]{Blow-ups and resolutions of strong K\"ahler with torsion metrics}
\author{Anna Fino and Adriano Tomassini}
\date{\today}
\address{Dipartimento di Matematica \\ Universit\`a di Torino\\
Via Carlo Alberto 10\\
10123 Torino\\ Italy} \email{annamaria.fino@unito.it}
\address{Dipartimento di Matematica\\ Universit\`a di Parma\\ Viale G.P. Usberti 53/A\\
43100 Parma\\ Italy} \email{adriano.tomassini@unipr.it}
\thanks{This work was supported by the Projects MIUR ``Riemannian Metrics and Differentiable Manifolds'',
``Geometric Properties of Real and Complex Manifolds'' and by
GNSAGA of INdAM}
\begin{abstract} On a compact complex manifold we study the behaviour of strong K\"ahler with torsion (strong KT) structures under small deformations
of the complex structure and the problem of extension of a strong
KT metric. In this context we obtain the analogous result of
Miyaoka extension theorem. Studying the blow-up of a strong KT
manifold at a point or along a complex submanifold, we prove that
a complex orbifold endowed with a strong KT metric admits a strong
KT resolution. In this way we obtain new examples of compact simply-connected strong KT manifolds.
\end{abstract}
\maketitle
\section{Introduction}
Let $(M,J,g)$ be a Hermitian manifold of complex dimension $n$. By
\cite{Ga2} there is a $1$-parameter family of canonical Hermitian
connections on $M$ which can be distinguished by properties of
their torsion tensor $T$. In particular, there exists a unique
connection $\nabla^B$ satisfying $\nabla^B g=0$, $\nabla^B J=0$
for which $g(X,T(Y,Z))$ is totally skew-symmetric. The resulting
3-form can then be identified with $JdF$, where $F (\cdot, \cdot)
= g(J \cdot, \cdot)$ is the fundamental 2-form associated to the
Hermitian structure $(J, g)$. This connection was used by Bismut
in \cite{Bi} to prove a local index formula for the Dolbeault
operator when the manifold is non-K\" ahler. The properties of
such a connection are related to what is called \lq\lq {\em
K\"ahler with torsion geometry}\rq\rq\, and if $Jd F$ is closed,
or equivalently if $F$ is $\partial \overline \partial$-closed,
then the Hermitian structure $(J, g)$ is strong KT and $g$ is
called a {\em strong KT} or a {\em{pluriclosed}} metric. The
strong KT metrics have also applications in type II string theory
and in 2-dimensional supersymmetric $\sigma$-models \cite{GHR,Str}
and have relations with generalized K\"ahler structures (see for
instance \cite{Gu,Hi2,AG,FT}).

The condition $\partial \overline \partial F =0$ is obviously
satisfied if $d F =0$, i.e. if $g$ is a K\"ahler metric. The
interesting strong KT metrics for us are those ones which are not
K\"ahler and therefore it is natural to investigate which
properties that hold for K\"ahler manifolds can be generalized in
the context of strong KT geometry. \newline In view of this, in
the present paper, we study in particular the behaviour of strong
KT structures under small deformations of the complex structure,
the blow-up of a strong KT manifold at a point or along a complex
submanifold and the problem of extension of a strong KT metric on
a complex manifold.

The theory about strong KT manifolds in complex dimension at least
three is completely different from that one on complex surfaces.
Indeed, on a complex surface a Hermitian metric satisfying the
strong KT condition is \lq \lq \emph{standard}\rq\rq\, in the
terminology of Gauduchon \cite{Ga} and there exists a standard
metric in the conformal class of any given Hermitian metric on a
compact manifold. Therefore on a complex surface the strong KT
condition is stable under small deformations of the complex
structure.

Examples of compact strong KT manifolds of complex dimension three
are given by nilmanifolds, i.e. compact quotients of nilpotent Lie
groups by uniform discrete subgroups (see \cite{FPS}). It is well
known that these manifolds are not formal in the sense of
\cite{Sul} and cannot admit any K\"ahler metric unless they are
tori (see \cite{BG,DGMS,Ha2}). More precisely, in \cite{FPS} it
was showed that if a nilmanifold of real dimension $6$ admits a
strong KT structure then the nilpotent Lie group has to be
$2$-step and therefore the nilmanifold has to be the total space
of a torus bundle over a torus.\newline One of the examples found
in \cite{FPS} is the {\em Iwasawa manifold}, which can also be
viewed as the total space of a ${\mathbb T}^2$-bundle over the
torus ${\mathbb T}^4$. In contrast with the Kodaira-Spencer
stability theorem \cite{KS} and the case of complex surfaces, in
Section \ref{deformations} we prove that on this manifold the
condition strong KT is not stable under small deformations of the
complex structure.

As in the K\"ahler case, in Section \ref{blowup} we prove that the
blow-up of a strong KT manifold at a point is still strong KT and
more in general

\smallskip
\noindent {\bf Proposition 3.2} {\em Let $M$ be a complex manifold
endowed with a strong KT metric $g$. Let $Y\subset M$ be a compact
complex submanifold. Then the blow-up $\tilde{M}_Y$ of $M$ along
$Y$ has a strong KT metric.}

\smallskip
\noindent This result allows, starting from a strong KT manifold,
to construct a new one.

The K\"ahler condition for a Hermitian metric can be characterized
in terms of positive currents. Indeed, Harvey and Lawson in
\cite{HL} proved that a compact complex manifold admits a K\"ahler
metric if and only if there is no non-zero positive
$(1,1)$-current which is the $(1,1)$-component of a boundary. By
\cite{Eg} also the strong KT condition can be studied in terms of
positive currents. A result of Miyaoka \cite{Mi} asserts that, if
a (compact) complex manifold $M$ has a K\"ahler metric in the
complement of a point, then $M$ is itself K\"ahler. By using the
extension result of \cite{AB} about positive or negative
plurisubharmonic currents, in Section \ref{extension} we prove the
analogous result of the Miyaoka's one for strong KT manifolds of
complex dimension $n$ with $n \geq 2$. As an application by using
the previous extension theorem we show the following

\smallskip
\noindent {\bf Theorem 4.3} {\em Let $M$ be a complex manifold of
of complex dimension $n\geq 2$.

\noindent If $M\setminus\{p\}$ admits a strong KT metric, then
there exists a strong KT metric on $M$.}

\smallskip

There are few examples of compact simply-connected strong KT
manifolds; as far as we know, they are given by real compact
semisimple Lie groups of even dimension \cite{SSTV}. Therefore it
is interesting to investigate for new compact simply-connected
strong KT manifolds. A natural way to obtain these is to consider
resolutions of orbifolds and the typical example of orbifold is
given by the quotient of a manifold by an action of a finite group
with non-identity fixed point sets of codimension at least two
(see \cite{Satake}). In Section \ref{sectionorbifolds} we study
resolutions of strong KT orbifolds. By using the result by
Hironaka \cite{Hiro} that a complex orbifold admits a resolution,
which is a obtained by a finite sequence of blow-ups and the
results obtained about blow-ups, we prove the following

\smallskip

\noindent {\bf Theorem 5.4} {\em Let $(M,J)$ be a complex orbifold
of complex dimension $n$ endowed with a $J$-Hermitian strong KT
metric $g$. Then there exists a strong KT resolution.}

\smallskip
\noindent In the last two sections we apply this result to complex
orbifolds constructed considering an action of a finite group on a
torus, on a product of a Kodaira-Thurston surface with the
$4$-dimensional torus $\T^4$ and on a product of two
Kodaira-Thurston surfaces. The strong KT resolutions that we get
have a simpler topology with respect to the one of the three
$8$-dimensional nilmanifolds. Indeed, in the case of the torus we
are able to construct a new compact simply-connected strong KT
manifold and in the other two cases the strong KT resolutions have
first Betti number equal to one. Resolutions for quotients of tori
have been already considered by Joyce in order to obtain
simply-connected compact manifolds with exceptional holonomy $G_2$
and $Spin (7)$ in \cite{J}. Moreover, in \cite{FM} and \cite{CFM}
orbifolds constructed starting with nilmanifolds have been
recently used in order to get simply-connected compact non-formal
symplectic manifolds.
\smallskip

\noindent {\em{Acknowledgements}} We would like to thank Alberto
Della Vedova, Simon Salamon and Stefano Trapani for useful
comments and stimulating conversations. We also would
like to thank the referee for valuable remarks.
\section{Small deformations of strong KT metrics}\label{deformations}
We will recall some basic definitions and fix some notation. Let
$(M,J)$ be a complex manifold of complex dimension $n$ and
decompose as usual $d=\partial + \overline{\partial}$. Let $g$ be
a Hermitian metric on $(M,J)$. The {\em fundamental $2$-form} $F$
is then defined by
$$
F(X , Y) = g(JX,Y)
$$
and has type $(1,1)$ relative to the complex structure $J$.
\begin{definition}
The Hermitian metric $g$ on $(M, g)$ is said to be {\em strong
K\"ahler with torsion}, or shortly, {\em strong KT}, if
\begin{equation}\label{definitionstrongKT}
\partial\overline{\partial}F=0\,.
\end{equation}
\end{definition}
\noindent Clearly, condition \eqref{definitionstrongKT} is weaker
than the K\"ahler one and the previous definition includes for us
the K\"ahler metrics.\newline In this section we will explore the
stability of strong KT metrics under infinitesimal deformations of
the complex structures.\newline By a well known result by
Kodaira-Spencer \cite{KS} the K\"ahler condition for a Hermitian
metric is stable under small deformations of the complex structure
underlying the K\"ahler structure. We will show that this does not
hold for strong KT structures in real dimension higher than or
equal to six. In real dimension $4$ a strong KT metric is standard
in the terminology of \cite{Ga} and therefore the result holds.

Indeed, we recall that a Hermitian structure $(J, g)$ on a
manifold $M$ of real dimension $2n$ is called {\em standard} in
the terminology of \cite{Ga} if
$$
\partial \overline \partial F^{n-1} =0,
$$
or equivalently if the Lee form $\theta = -J* d * F $ is
co-closed, where
$$
F^{n -1}= \underbrace{F \wedge \ldots \wedge
F}_{(n-1)-{\mbox{times}}}\,.
$$
In particular, if $\theta =0$ the Hermitian structure is said to
be {\em balanced}.

By \cite{Ga} for a compact complex manifold a standard metric can
be found in the conformal class of any given Hermitian metric.
Since on a complex surface a strong KT metric is standard, a small
deformation of the complex structure on a complex surface
preserves the strong KT condition. In higher dimensions the strong
KT condition is not anymore equivalent to the standard one, in
fact by \cite{AI} a strong KT metric is standard only if
$$
\vert dF \vert^2 = (n- 1) \vert \theta \wedge F \vert^2,
$$
where $\theta$ is the Lee form of the Hermitian structure $(J, g)$
and by $\vert \cdot \vert$ we denote the norm of the form.
Therefore, if $n > 2$ a strong KT metric is not necessarily
standard. An example of compact strong KT manifold of complex
dimension three is given by the Iwasawa manifold $\Iw$, which is
the compact quotient of the complex Heisenberg group
$$
H_3^{\C} = \left\{ \left( \begin{array}{ccc} 1&z_1&z_3\\ 0&1&z_2\\
0&0&1 \end{array} \right)\,\,\vert\,\, z_j \in \C \, , j=1,2,3
\right\}
$$
by the uniform discrete subgroup $\Gamma$ for which $z_j$ are
Gaussian integers. By \cite[Theorem 1.2]{FPS} for this manifold
and more in general for any nilmanifold of complex dimension three
the strong KT condition depends only on the underlying complex
structure. This allows us in contrast with the case $n =2$ to
prove the following
\begin{theorem}
On the Iwasawa manifold $\Iw = \Gamma \backslash H_3^{\C}$ the
condition for a Hermitian metric to be strong KT is not stable
under small deformations of the complex structure underlying the
strong KT structure.
\end{theorem}
In order to prove the theorem we will construct an explicit
deformation of a complex structure underlying a strong KT
structure that does not remain strong KT.

Let ${\mathfrak n}_{t,s}$ be the family of $2$-step nilpotent Lie
algebras with structure equations
$$
\left\{
\begin{array}{l}
d e^i =0, \quad i = 1, \ldots, 4\,,\\[5pt]
d e^5 = t(e^{1} \wedge e^{2} + 2 \, e^{3} \wedge e^{4}) + s (e^{1} \wedge e^{3} - e^{2} \wedge e^{4}),\\[5pt]
d e^6 = s (e^{1} \wedge e^{4} + e^{2} \wedge e^{3}),
\end{array}
\right.
$$
with $t$ and $s$ real numbers and $s \neq 0$. This family was
already considered in \cite{FG} for Hermitian structures whose
Bismut connection has holonomy in $SU(3)$ and it was proved that
for any $t$ and $s \neq 0$ the Lie algebra ${\mathfrak n}_{t,s}$
is isomorphic to the Lie algebra of the complex Heisenberg group
$H_3^{\C}$ with structure equations
$$
\left\{
\begin{array}{l}
d e^i =0, \quad i = 1, \ldots, 4\,,\\[5pt]
d e^5 =e^{1} \wedge e^{3} - e^{2} \wedge e^{4},\\[5pt]
d e^6 = e^{1} \wedge e^{4} + e^{2} \wedge e^{3},
\end{array}
\right.
$$
(compare also \cite[Example 6.1 and 6.5]{Lau}).

Take the almost complex structure $J$ on ${\mathfrak n}_{t,s}$
given by
\begin{equation} \label{complrealHeis}
J e^1 = e^2, \, J e^3 = e^4, \, J e^5 = J e^6.
\end{equation}
For the associated $(1,0)$-forms
$$
\varphi^1 = e^1 + i e^2\,, \quad \varphi^2 = e^3 + i e^4\,, \quad \varphi^3
= e^5 + i e^6,
$$
we have that
$$
\begin{array} {l}
d \varphi^i =0\,, \quad i =1,2\,,\\[5pt]
d \varphi^3 = - \frac 12 i\, t \, (\varphi^1 \wedge \overline
\varphi^1 + 2 \varphi^2 \wedge \overline \varphi^2) + s \,
\varphi^1 \wedge \varphi^2,
\end{array}
$$
and therefore $J$ is integrable.

In this way the Iwasawa manifold $\Iw = \Gamma \backslash
H_3^{\C}$ is endowed with a family of complex structures
$J_{t,s}$, with $t, s \in \R$ and $s \neq 0$.

Note that for $t =0$ and $s =1$ the complex structure $J$
coincides with the bi-invariant complex structure $J_0$ on the
complex Heisenberg group. The complex structure $J_0$ cannot admit
any compatible strong KT metric, since otherwise it has to be
balanced and by \cite{FPS} the balanced condition is complementary
to the strong KT one.

We will show that the Iwasawa manifold $(\Iw, J_{t,s})$ admits a
strong KT metric compatible with $J_{t, s}$ if and only if $t^2
=s^2$. Indeed, by \cite[Lemma 1.3]{FPS}, if $g$ is a
left-invariant Riemannian metric compatible with $J_{t,s}$, $g$ is
strong KT if and only if
$$
\partial \overline \partial\, (\varphi^3 \wedge \overline \varphi^3) =0 = (t^2 - s^2)
\varphi^1 \wedge \varphi^2 \wedge \overline \varphi^1 \wedge
\overline \varphi^2.
$$
By \cite[Proposition 3.1]{Ug} and \cite{FG} if there exists a
non-left-invariant strong KT metric compatible with $J_{t,s}$,
then there is also a left-invariant one. Therefore this is only
possible if $t^2 = s^2$.

Thus if $t = s =1$ the Iwasawa manifold has a strong KT metric $g$
compatible with $J_{1,1}$, but for any $t \neq s \neq 1$ there
exists no a strong KT metric compatible with the complex structure
$J_{t,s}$.

In a similar way one can show that generically on a nilmanifold of
complex dimension three the condition strong KT is not stable
under small deformations of the underlying complex structure, but
one can also construct a family of strong KT structures. For
instance consider the family of Lie algebras ${\mathfrak n}_{t,s}$
with structure equations
$$
\left\{
\begin{array}{l}
d e^i=0\,, \quad i =1, \ldots, 5\,,\\[5pt]
d e^6 = t^2 \, e^1 \wedge e^2 + ts \, (e^1 \wedge e^4 - e^2 \wedge
e^3) + s^2 \, e^3 \wedge e^4.
\end{array}
\right.
$$
For any real numbers $t,s$ such that $t^2 + s^2 \neq 0$ the Lie
algebra ${\mathfrak n}_{t,s}$ is isomorphic to the Lie algebra of
$H_3 \times \R^3$, where $H_3$ is the real $3$-dimensional
Heisenberg Lie group. Moreover, for any $t,s$ the complex
structure $J$ defined by \eqref{complrealHeis} gives rise to a
strong KT structure.
\section{Blow-up of strong KT manifolds}\label{blowup}
We start by proving that the blow-up of a strong KT manifold at a
point is still strong KT, as in the K\"ahler case (see for example
\cite{BL}).
\begin{prop}\label{blowuppoint}
Let $(M,J,g)$ be a strong KT manifold of complex dimension $n$ and
$\tilde M_p$ be the blow-up of $M$ at a point $p\in M$. Then
$\tilde M_p$ admits a strong KT structure.
\end{prop}
\begin{proof}
Let $z = (z_1, \ldots ,z_n)$ be holomorphic coordinates in an open
set $U$ centered around the point $p \in M$. We recall that the
blow-up $\tilde M_p$ of $M$ is the complex manifold obtained by
adjoining to $M \setminus \{ p \}$ the manifold
$$
\tilde U = \{ (z, l) \in U \times \CP^{n - 1} \, \vert \, z \in l
\}
$$
by using the isomorphism
$$
\tilde U \setminus \{ z =0 \} \cong U \setminus \{ p \}
$$
given by the projection $ (z, l) \to z$. In this way there is a
natural projection $\pi: \tilde M_p \to M$ extending the identity
on $M \setminus \{ p \}$ and the exceptional divisor $\pi^{-1}(p)$
of the blow-up is naturally isomorphic to the complex projective
space $\CP^{n - 1}$.

If we denote by $F$ the fundamental $2$-form associated with the
strong KT metric $g$, then the $2$-form $\pi^* F$ is $\partial
\overline \partial$-closed since $\pi$ is holomorphic, but it is
not positive definite on $\pi^{-1} (M \setminus \{ p \})$. As in
the K\"ahler case, let $h$ be a ${\mathcal C}^{\infty}$-function
having support in $U$, i.e. $0 \leq h \leq 1$ and $h=1$ in a
neighborhood of $p$. On $U \times (\C^n \setminus \{ 0 \})$
consider the $2$-form
$$
\gamma = i \partial \overline \partial \left( (p_1^* h) p_2^* \log
\vert \vert \cdot\vert \vert^2 \right)\,,
$$
where $p_1$ and $p_2$ denote the two projections of $U \times
(\C^n \setminus \{ 0 \})$ on $U$, $\C^n\setminus\{0\}$
respectively.\newline Let $\psi$ be the restriction of $\gamma$ to
$\tilde M_p$. Then there exists a small enough real number
$\epsilon$ such that the $2$-form $\tilde F = \epsilon \psi +
\pi^* F$ is positive definite. Since $\tilde F$ is $ \partial
\overline \partial$-closed, it defines a strong KT metric on the
blow-up $\tilde{M}_p$.
\end{proof}
\noindent As a consequence of Proposition \ref{blowuppoint}, it is possible
to construct new examples of strong KT manifolds by blowing-up a
given strong KT manifold $M$ at one or more points. Moreover, the
homology groups of the two manifolds $M$ and $\tilde M_p$ are
related by
$$
H_i (\tilde M_p) = H_i (M) \oplus H_i (\C P^{n - 1}), \quad i \geq
1.
$$
Note that in view of Proposition \ref{blowupextension}, the
blow-up of the strong KT nilmanifolds given in \cite{FPS} cannot
admit any K\"ahler structure.\newline Proposition
\ref{blowuppoint} can be generalized to the blow-up of a strong KT
manifold along a compact complex submanifold. Indeed
\begin{prop} \label{blowupsubmanifold}
Let $M$ be a complex manifold endowed with a strong KT metric $g$.
Let $Y\subset M$ be a compact complex submanifold. Then the
blow-up $\tilde{M}_Y$ of $M$ along $Y$ has a strong KT metric.
\end{prop}
\begin{proof}
Let $\pi :\tilde{M}_Y\to M$ be the holomorphic projection. By
construction $\pi:\tilde{M}\setminus\pi^{-1}(Y)\to M\setminus Y$
is a biholomorphism and $\pi^{-1}(Y)\cong\Proj(\mathcal{N}_{Y\vert
M})$, where $\Proj(\mathcal{N}_{Y\vert M})$ is the projectified of
the normal bundle of $Y$. Let $F$ be the fundamental $2$-form of
the strong KT metric on $M$. There exists a holomorphic line
bundle $L$ on $\tilde{M}_Y$ such that $L$ is trivial on
$\tilde{M}_Y\setminus\pi^{-1}(Y)$ and such that its restriction to
$\pi^{-1}(Y)$ is isomorphic to
$\mathcal{O}_{\Proj(\mathcal{N}_{Y\vert M})}(1)$. \newline Let $h$
be a Hermitian structure on
$\mathcal{O}_{\Proj(\mathcal{N}_{Y\vert M})}(1)$ and $\omega$ be
the corresponding Chern form. Let $\left\{U_i\right\}_{i\in I}$ be
an open covering of $M$ which trivializes the line bundle $L$. By
using a partition of unity subordinate to
$\left\{U_i\right\}_{i\in I}$, it follows that the metric $h$ can
be extended to a metric structure $\hat{h}$ on $L$, in such a way
that $\hat{h}$ is the flat metric structure on the complement of a
compact neighborhood $W$ of $Y$ induced by the trivialization of
$L$ on $\tilde{M}_Y\setminus\pi^{-1}(Y)$. Therefore, the Chern
curvature $\hat{\omega}$ of $L$ vanishes on $M\setminus W$ and
$\hat{\omega}\vert_{\Proj(\mathcal{N}_{Y\vert M})}=\omega$.
\newline Hence, since $Y$ is compact, there exists $\epsilon
\in\R,\,\epsilon >0$, small enough, such that
$$
\tilde{F} =\pi^* F +\epsilon\,\hat{\omega}
$$
is positive definite. Moreover, $\partial\overline{\partial}\,
\tilde{F}=0$, so that $\tilde{F}$ gives rise to a strong KT metric on
$\tilde{M}_Y$.
\end{proof}
\noindent By \cite{Bl}Ê the cohomology groups of the two manifolds $M$ and
$\tilde M_{Y}$ are related by
$$
H^*(\tilde M_Y) = \pi^* H^* (M) \oplus H^*
(\Proj(\mathcal{N}_{Y\vert M})) / \pi^* H^*(Y)
$$
and therefore for the corresponding Poincar\'e polynomials we have
$$
P_{\tilde M_Y} (t) = P_{M} (t) + P_Y (t) \left( \sum_{j =1}^{n - k
- 1} t^{2j} \right),
$$
where $k = \dim_{\C} Y$. In particular, for the first Betti number
we get $$b_1 (\tilde M_Y) = b_1 (M).
$$
\section{Extension of strong KT metrics} \label{extension}
We start by fixing some notation and recalling some known facts on
positive currents. For our purposes it is enough to consider an
open set $\Omega\subset\C^n$. \newline Denote by
$\Lambda^{p,q}(\Omega)$ (respectively by ${\mathcal
D}^{p,q}(\Omega)$) the space of $(p,q)$-forms (respectively
$(p,q)$-forms with compact support) on $\Omega$. On ${\mathcal
D}^{p,q}(\Omega)$ consider the ${\mathcal C}^\infty$-topology. By
definition, the {\em space of currents} of {\em bi-dimension}
$(p,q)$ or of {\em bi-degree} $(n-p,n-q)$ is the topological dual
${\mathcal D}'_{p,q}(\Omega)$ of ${\mathcal D}^{p,q}(\Omega)$. A
current of bi-dimension $(p,q)$ on $\Omega$ can be identified with
a $(n-p,n-q)$-form on $\Omega$ with coefficients distributions.
The {\em support} of a current $T\in {\mathcal D}'_{p,q}(\Omega)$,
denoted by $\hbox{\rm supp}(T)$, is the smallest closed set $C$
such that the restriction of $T$ to ${\mathcal
D}^{p,q}(\Omega\setminus C)$ is zero.
A current $T\in{\mathcal D}'_{p,q}(\Omega)$ is said to be of {\em
order} $0$ if its coefficients are measures and is said to be {\em
normal} if $T$ and $dT$ are currents of order $0$. \newline A
current $T$ of bi-dimension $(p,p)$ is said to be {\em real} if
$T(\varphi) =T(\overline{\varphi})$, for any $\varphi\in {\mathcal
D}^{p,q}(\Omega)$. Therefore, if $T\in{\mathcal D}'_{p,p}(\Omega)$
is real, then we may write
$$
T =\sigma_{n-p}\sum_{I,\overline{J}} T_{I\overline{J}} dz_I\wedge
d\overline{z}_J\,,
$$
where $\sigma_{n-p}=\frac{i^{(n-p)^2}}{2^{(n-p)}}$,
$T_{I\overline{J}}$ are distributions on $\Omega$ such that
$T_{J\overline{I}}=\overline{T}_{I\overline{J}}$ and $I$, $J$ are
multi-indices of length $n-p$, $I=(i_1,\ldots ,i_{n-p})$,
$dz_I=dz_{i_1}\wedge\cdots \wedge dz_{i_{n-p}}$.

A real current $T\in{\mathcal D}'_{p,p}(\Omega)$ is {\em positive}
if,
$$
T(\sigma_p\,\varphi^1\wedge\cdots\wedge\varphi^p\wedge\overline{\varphi}^1\wedge\cdots\wedge\overline{\varphi}^p)
\geq 0
$$
for any choice of $\varphi^1,\ldots ,\varphi^p\in{\mathcal
D}^{1,0}(\Omega)$, where $\sigma_p=\frac{i^{p^2}}{2^p}$. A current
$T$ is said to be {\em strictly positive} if $\varphi^1 \wedge
\cdots
 \wedge \varphi^p\neq 0$ implies
$T(\sigma_p\,\varphi^1\wedge\cdots\wedge\varphi^p\wedge\overline{\varphi}^1\wedge\cdots\wedge\overline{\varphi}^p)>0$.
\newline Recall that if $T$ is a positive current of bi-degree
$(p,p)$, then $T$ is of order $0$. \newline A real current $T$ of
bi-dimension $(p,p)$ on $\Omega$ is said to be {\em negative} if
the current $-T$ is positive and {\em plurisubharmonic} if
$i\,\partial\overline{\partial} T$ is positive.

If $F$ is the fundamental $2$-form of a Hermitian structure on a
complex manifold $M$, then $F$ corresponds to a real strictly
positive current of bi-degree $(1,1)$. In particular, if the
Hermitian structure is strong KT, then the corresponding current
is $\partial \overline \partial$-closed.

An important class of $\partial \overline \partial$-closed
currents is given by the $(p,p)$-components of a boundary. We
recall that a current $T$ of bi-degree $(p,p)$ is called the {\em
$(p,p)$-component of a boundary} if there exists a real current
$S$ of bi-degree $(p,p-1)$ such that $T = \partial \overline S +
\overline \partial S$. In \cite{HL} Harvey and Lawson proved that
a compact complex manifold has a K\"ahler metric if and only if
there is no non-zero positive current of bi-degree $(1,1)$ which
is the $(1,1)$-component of a boundary. By \cite{Eg} this
characterization of the K\"ahler condition can be generalized in
the context of strong KT geometry showing that a compact complex
manifold admits a strong KT metric if and only if there is no
non-zero positive current of bi-degree $(1,1)$ which is $\partial
\overline \partial$-exact.

In \cite{Mi} Miyaoka showed that if a complex manifold $M$ has a
K\"ahler metric in the complement of a point, then the manifold
$M$ itself is K\"ahler.\newline In order to prove a similar result
for strong KT structures we need to recall the following extension
theorem (see \cite[Main Theorem 5.6]{AB})
\begin{theorem}\label{extension1}
Let $Y$ be an analytic subset of $\Omega\subset\C^n$. If $T$ is a
plurisubharmonic, negative current of bi-dimension $(p,p)$ on the
complement $\Omega \setminus Y$ of $Y$ in $\Omega$ and $\dim_\C Y<
p$, then there exists the simple (or trivial) extension $T^0$ of
$T$ across $Y$ and $T^0$ is plurisubharmonic.
\end{theorem}
\noindent If $T =\sigma_{n-p}\sum_{I,\overline{J}}
T_{I\overline{J}} dz_I\wedge d\overline{z}_J\,,$ on
$\Omega\setminus Y$, with $T_{I\overline{J}}$ measures, then the
current $T^0$ on $\Omega$ is defined by extending the
$T_{I\overline{J}}$ to zero on $Y$.\newline
Finally, we recall (see e.g. \cite[Corollary 2.11 p.163]{Dem}) the
following corollary of the Theorem of support \cite[Theorem 2.10
p.163]{Dem}
\begin{theorem}\label{support}
Let $T$ be normal current of bi-dimension $(p,p)$ on
$\Omega\subset\C^n$. If $\hbox{\rm supp}(T)$ is contained in an
analytic subset $Y$ of $\Omega$ such that $\dim_\C Y< p$, then
$T=0$.
\end{theorem}
\noindent By using the previous results we are ready to prove the
following
\begin{theorem}\label{extensionddbar}
Let $M$ be a complex manifold of of complex dimension $n\geq 2$.

\noindent If $M\setminus\{p\}$ admits a strong KT metric, then
there exists a strong KT metric on $M$.
\end{theorem}
\noindent The proof of the Theorem \ref{extensionddbar} is a
consequence of the following
\begin{prop}\label{miyaoka}
Let $F$ be the fundamental $2$-form of a strong KT metric on
$\B^n(r)\setminus\{0\}$, $n\geq 2$. Then there exist $0<R\leq r$
and $\hat{F}\in\Lambda^{1,1}(\B^n(R))$ such that
\begin{enumerate}
\item[i)] $\hat{F}$ is the fundamental $2$-form of a strong KT
metric on $\B^n(R)$,\smallskip \item[ii)] $\hat{F} =F$ on
$\B^n(R)\setminus \B^n(\frac23R)$.\smallskip
\end{enumerate}
\end{prop}
\begin{proof} A key tool in the proof of the Proposition is the following result by
\cite[Theorem 1.15]{B} on $\partial\overline{\partial}$-closed
currents, which is based on an argument given by Siu
(\cite[p.121]{Si}), for $d$-closed $l$-currents with measure
coefficients.
\begin{theorem}\label{aeppli}
Let $T$ be a current of bi-degree $(h,k)$ on $\Omega$. If $T$ is
of order $0$ and $i\,\partial\overline{\partial} T =0$, then,
locally,
$$
T = \partial G + \overline{\partial} H\,,
$$
for suitable currents $G$ and $H$ with locally integrable
functions as coefficients.
\end{theorem}
\noindent For the sake of completeness, we will give the proof of the Theorem 4.5 (see \cite[Theorem 1.15]{B} and
\cite[p.121]{Si}).\smallskip

\noindent {\em Proof of Theorem \ref{aeppli}.}
Consider
$$
\Lambda^{1,0}(\Omega)\oplus
\Lambda^{0,1}(\Omega)\stackrel{\partial+\overline{\partial}}\longrightarrow
\Lambda^{1,1}(\Omega)\stackrel{i\partial\overline{\partial}}\longrightarrow\Lambda^{2,2}(\Omega).
$$
Then, according to the Theorem of Hodge for elliptic complexes,
(see e.g. \cite[p.235]{Va}), the differential operator
$$
\Box =
(\partial+\overline{\partial})(\partial+\overline{\partial})^*(\partial+\overline{\partial})(\partial+\overline{\partial})^*+
(i\partial\overline{\partial})^*(i\partial\overline{\partial})
$$
is elliptic. In view of \cite[Theorem 7.1.20]{Ho}, there exists a
fundamental solution $E$ of the differential operator $\Box$ given
by
$$
E = E_0-Q(z)\log\vert z\vert\,,
$$
where $E_0$ is a matrix of homogeneous distributions of degree
$4-2n$, smooth in $\C^n\setminus\{0\}$, $Q$ is a matrix of
polynomials which vanishes identically for $2n>4$ and it is constant for $2n=4$.\newline Set
$$
L=(\partial+\overline{\partial})^*(\partial+\overline{\partial})(\partial+\overline{\partial})^*
$$
and let $\lambda$ be a ${\mathcal C}^\infty$-function, with
compact support contained in $\Omega$ and such that $\lambda(z)=1$
on a ball $U'\subset\Omega$, $0\in U'$. Then, we have
\begin{eqnarray*}
\lambda T & = & \Box(E*\lambda T)\\
{} & = & (\partial +\overline{\partial})(L(E*\lambda T))+
\overline{\partial}^*\partial^*E* \partial\overline{\partial}(\lambda T)\\
{} & = & (\partial +\overline{\partial})(L(E*\lambda T))+
\overline{\partial}^*\partial^*E*(\partial\overline{\partial}\lambda\wedge
T - \overline{\partial}\lambda \wedge\partial T +\partial
\lambda\wedge\overline{\partial}T +
\lambda\partial\overline{\partial}T)\\
{} & = & (\partial +\overline{\partial})(L(E*\lambda T))+
\overline{\partial}^*\partial^*E*(\partial\overline{\partial}\lambda\wedge
T - \overline{\partial}\lambda \wedge\partial T +\partial
\lambda\wedge\overline{\partial}T)\,.
\end{eqnarray*}
By a direct computation, it turns out that the current $\overline{\partial}^*\partial^*E$
has locally integrable functions as coefficients. Hence, the
coefficients of the current
$$
\overline{\partial}^*\partial^*E*(\partial\overline{\partial}\lambda\wedge
T - \overline{\partial}\lambda \wedge\partial T +\partial
\lambda\wedge\overline{\partial}T)
$$
are locally integrable functions, since they are obtained as
convolutions of locally integrable functions with
measures.\newline Since $\lambda =1$ on $U'$,
$\partial\overline{\partial}\lambda\wedge T -
\overline{\partial}\lambda \wedge\partial T +\partial
\lambda\wedge\overline{\partial}T$ vanishes identically on $U'$.
Therefore, by \cite[Theorem 4.2.5]{Ho}, it follows that
$$
\mbox{\rm sing
supp}\,(\overline{\partial}^*\partial^*E*(\partial\overline{\partial}\lambda\wedge
T - \overline{\partial}\lambda \wedge\partial T +\partial
\lambda\wedge\overline{\partial}T))\subset \Omega\setminus U'
$$
where $\mbox{\rm sing supp}$ denotes the singular support of a
current, i.e. the complement in $\Omega$ of the open set $A$ such
that the restriction of the current to $A$ is smooth. Hence,
$$
\overline{\partial}^*\partial^*E*(\partial\overline{\partial}\lambda\wedge
T - \overline{\partial}\lambda \wedge\partial T +\partial
\lambda\wedge\overline{\partial}T)
$$
is ${\mathcal C}^\infty$ on $U'$. Furthermore, it is
$\partial\overline{\partial}$-closed. Consequently, there exist a
$(h-1,k)$-form $\phi$ and a $(h,k-1)$-form $\psi$ on $U'$, such
that
$$
\overline{\partial}^*\partial^*E*(\partial\overline{\partial}\lambda\wedge
T - \overline{\partial}\lambda \wedge\partial T +\partial
\lambda\wedge\overline{\partial}T) =\partial \phi +
\overline{\partial}{\psi}\,.
$$
Therefore, on $U'$, we can write
$$
T=\partial G +\overline{\partial}H
$$
where
\begin{equation}\label{smooth}
G=L(E*\lambda T)+\phi\,,\qquad H= L(E*\lambda T)+\psi\,.
\end{equation}
\hfill $\Box$
\bigskip

\noindent Now, let us start with the proof of the
Proposition.\newline By hypothesis, $F$ is a positive
$\partial\overline{\partial}$-closed $(1,1)$-form on
$\B^n(r)\setminus\{0\}$. Let $T=-F$; then $T$ is a real (strictly)
negative $\partial\overline{\partial}$-closed current of bi-degree
$(1,1)$ on $\B^n(r)\setminus\{0\}$. In view of the result in
\cite[Main Theorem 5.6]{AB} (see Theorem \ref{extension1} above),
applied to the case $Y=\{ 0 \}$, the simple extension $T^0$ of $T$
on the ball $\B^n(r)$, defined by
$$
T^0(\varphi)=\int_{\B^n(r)\setminus\{0\}}F\wedge\varphi\,,\qquad
\forall \varphi\in\mathcal{D}^{n-1,n-1}(\B^n(r)),
$$
is negative on $\B^n(r)$. \newline Consider now the current
$i\,\partial\overline{\partial} T^0$. We have that
$i\,\partial\overline{\partial} T^0$ is positive and consequently
it is of order $0$. Moreover,
$d\left(i\,\partial\overline{\partial} T^0 \right)=0$. Therefore,
$i\,\partial\overline{\partial} T^0$ is a normal current of
bi-degree $(2,2)$ on the ball $\B^n(r)$. Hence, by the Corollary
of the support Theorem (see Theorem \ref{support} above) we obtain
that
$$
i\,\partial\overline{\partial} T^0=0\,
$$
on $\B^n(r)$. \newline Therefore, $T^0$ is a negative
$\partial\overline{\partial}$-closed current of bi-degree $(1,1)$
on the ball $\B^n(r)$. Moreover, the coefficients of $T^0$ are
measures. Observe that $T^0$ is smooth on $\B^n(r)\setminus\{0\}$.
\newline Set $F^0=-T^0$. Then $F^0$ is clearly a real positive
$\partial\overline{\partial}$-closed current of bi-degree $(1,1)$
on $\B^n(r)$ and it is strictly positive on $\B^n(r) \setminus \{
0 \}$.\newline In view of \cite[Theorem 1.15]{B} (see Theorem
\ref{aeppli} above) and reality of $F^0$, we may write
$$
F^0 =\partial G+\overline{\partial}\,\overline{G}\,,\quad
\hbox{\rm on}\, \, \B^n(R)
$$
for some $0<R\leq r$, where $G$ is a current of bi-degree $(0,1)$
whose coefficients are locally integrable functions on
$\B^n(R)$.\newline As a consequence of \eqref{smooth}, $G$ is in
fact smooth on $\B^n(R)\setminus\{0\}$. Indeed,
$$
G=L(E*\lambda F_0)+\phi
$$
and the fundamental solution $E$ of $\Box$ and the current $F_0$
are smooth on $\B^n(R)\setminus\{0\}$. Again by \cite[Theorem
4.2.5]{Ho}, it follows that
$$
\mbox{\rm sing supp}(E*\lambda F_0)\subset\mbox{\rm sing supp}\, E
+ \mbox{\rm sing supp}\,\lambda F_0\,.
$$
Therefore, $G$ is smooth on $\B^n(R)\setminus\{0\}$.\newline Now
we are going to define a strong KT metric with fundamental
$2$-form $\hat{F}$ on $\B^n(R)$ as in the statement. \newline Set
$$
G = \sum_{j=1}^nu_j d\overline{z}_j\,,
$$
where $u_i$ are locally integrable on $\B^n(R)$ and smooth on
$\B^n(R)\setminus\{0\}$. \newline Let $\rho :\C^n\to \R$ be a
non-negative ${\mathcal C}^\infty$-function on $\C^n$ such that
\begin{enumerate}
\item[a)] $\rho$ is radial and $\hbox{\rm supp}\left(
\rho(z)\right)\subset \B^n(\frac23R)$\smallskip \item[b)]
$\rho(z)= 1\,,\quad \forall
z\in\overline{\B^n(\frac13R)}$\smallskip \item[c)]
$\int_{\C^n}\rho (z)dz =1$\smallskip
\end{enumerate}
Define
$$
\tilde{u}_{j\,\epsilon}(z)
=\int_{\C^n}u_j\left(z-\epsilon\rho(z)\zeta\right)\rho(\zeta)\,d\zeta\,,\qquad
j=1,\ldots ,n\,.
$$
By using the conditions a), b) and c) it can be checked that, for
any $j =1,\ldots ,n$, $\tilde{u}_{j \, \epsilon}(z)$ is a
${\mathcal C}^\infty$-function on $\B^n(R)$ such that
$$
\tilde{u}_{j \,\epsilon}(z) =u_j (z)\, \quad\hbox{\rm on}\,\,
\B^n(R)\setminus\B^n\left(\frac23R \right)\,.
$$
Now, if we set
$$
G_\epsilon =\sum_{j=1}^n\tilde{u}_{j\,\epsilon}\,
d\overline{z}_j\,,
$$
then
$$
\tilde{F}_\epsilon = \partial G_\epsilon
+\overline{\partial}\,\overline{G}_\epsilon
$$
is a real $\partial\overline{\partial}$-closed $(1,1)$-form on
$\B^n(R)$ such that $\tilde{F}_\epsilon = F$ on
$\B^n(R)\setminus\B^n(\frac23R)$.\newline Note that, for
$\epsilon$ small enough, $\tilde{F}_\epsilon$ is strictly positive
on $\B^n(R)\setminus\{0\}$ and positive on $\B^n(R)$. Therefore,
in order to get the strict positivity on the whole ball $\B^n(R)$
we need to perturb $\tilde{F}_\epsilon$. To such a purpose, let $h
:\C^n\to \R$ be a non-negative ${\mathcal C}^\infty$-function on
$\C^n$ such that
\begin{enumerate}
\item[i)] $\hbox{\rm supp}\left(h(z)\right)\subset
\B^n(\frac13R)$\smallskip \item[ii)] $h(z)= 1\,,\quad \forall
z\in\overline{\B^n(\frac16R)}$.
\end{enumerate}
Then define
$$
\hat{F_\epsilon} = \partial\left(G_\epsilon +
\frac{i}{2}c\,\overline{\partial}\left(h(z)\vert z\vert^2\right)
\right) + \overline{\partial}\,\overline{\left(G_\epsilon +
\frac{i}{2}c\,\overline{\partial}\left(h(z)\vert z\vert^2\right)
\right)}\,,
$$
where $c$ is a real number. We immediately obtain
$$
\hat{F_\epsilon} = \tilde{F}_\epsilon
+ic\,\partial\overline{\partial}\left(h(z)\vert z\vert^2\right)\,.
$$
and consequently $\hat{F}_\epsilon =\tilde{F}_\epsilon$ on
$\B^n(R)\setminus\B^n(\frac13R)$. Finally, $\hat{F}_\epsilon$ is
strictly positive on $\B^n(\frac16R)$ and, by choosing the
positive real number $c$ small enough, we get that
$\hat{F}_\epsilon$ is also strictly positive on
$\overline{\B^n(\frac13R)}\setminus\B^n(\frac16R)$. \newline
Therefore, $\hat{F}_\epsilon$ is a real, (strictly) positive,
$\partial\overline{\partial}$-closed, $(1,1)$-form on $\B^n(R)$
and thus it gives rise to a strong KT metric we were looking for.
\end{proof}

\noindent {\em Proof of Theorem \ref{extensionddbar}.} Let $V$ be
a disc around $p \in M$. Then by Proposition \ref{miyaoka}, there
exist
 a smaller disc $U\subset V$ and a positive,
$\partial\overline{\partial}$-closed $(1,1)$-form $\hat{F}$ on $V$
such that $\hat{F}=F$ on $V\setminus U$. Therefore, the
$(1,1)$-form defined by
$$
\left\{
\begin{array}{l}
F_q\quad {\mbox {if}} \quad q \in M\setminus U,\\[3pt]
\hat{F}_q\quad {\mbox {if}} \quad q\in V
\end{array}
\right.
$$
is the fundamental $2$-form of a strong KT metric on $M$. \qed
\smallskip

\begin{rem}
{\rm If $n=1$, then any Hermitian metric is K\"ahler and,
consequently, for $n=1$, the proof of Theorem \ref{extensionddbar}
follows at once.}
\end{rem}
\noindent As an application of Theorem \ref{extensionddbar} we can
prove the following
\begin{theorem}\label{blowupextension}
Let $M$ be a complex manifold of complex dimension $n\geq 2$ and
$\tilde{M}$ be the blow-up of $M$ at a point $p \in M$. Then
$\tilde{M}$ has a strong KT metric if and only if $M$ admits a
strong KT metric.
\end{theorem}
\begin{proof}
Assume that $\tilde{M}$ has a strong KT metric. Let
$E=\pi^{-1}(p)$. Then \newline $\pi :\tilde{M}\setminus E\to
M\setminus\{p\}$ is a biholomorphism. Therefore $M\setminus\{p\}$
has a strong KT metric. By Theorem \ref{extensionddbar}, $M$ has a
strong KT metric.\newline The other implication is given by
Proposition \ref{blowuppoint}.
\end{proof}
\section{Strong KT orbifolds and resolutions} \label{sectionorbifolds}
Orbifolds are a special class of singular manifolds and they have
been used by Joyce in \cite{J} to construct compact manifolds with
special holonomy and in \cite{FM} to obtain non-formal symplectic
compact manifolds.\newline We start by recalling the following
(see e.g. \cite{J})
\begin{definition}
A {\em complex orbifold} is a singular complex manifold $M$ of
dimension $n$ such that each singularity $p$ is locally isomorphic
to $U/G$, where $U$ is an open set of $\C^n$, $G$ is a finite
subgroup of $\GL(n,\C)$ acting linearly on $U$ with the only one
fixed point $p$. Moreover, the set $S$ of singular points of $M$
of the orbifold $M$ has real codimension at least two.
\end{definition}
\noindent A very easy method to construct complex orbifolds is to
consider a holomorphic action of a finite group $G$ on a manifold
$M$, with non-identity fixed point sets of real codimension at
least two. The quotient $M /G$ is thus by definition a complex
orbifold.

Since orbifolds have a mild form of singularities, many good
properties for manifolds also hold for the orbifolds. For
instance, the notions of smooth $r$-forms and $(p,q)$-forms make
sense on complex orbifolds. The De Rham and Dolbeault cohomologies
are well-defined for orbifolds and they have many of the usual
properties that they have in the case of complex manifolds.

More precisely, an {\em $r$-orbifold differential form} on a
complex orbifold $(M, J)$ is an $r$-differential form on $M$ that
is $G$-invariant in any chart $U/G$ of $M$ and a differential
operator $d:\Lambda^r_{ orb}(M)\to \Lambda^{r+1}_{ orb}(M)$ is
defined on the complex $\Lambda_{ orb}(M)$ of orbifold
differential forms on $M$. For the complex space
$\Lambda^r_{orb}(M)\otimes \C$ we have
$$
\Lambda^r_{orb}(M)\otimes \C
=\displaystyle\bigoplus_{p+q=r}\Lambda^{p,q}_{orb}(M)\,.
$$
The elements of $\Lambda^{p,q}_{orb}(M)$ are called {\em
$(p,q)$-forms}, and, according to the above decomposition, the
differential $d$ splits as $d=\partial + \overline{\partial}$, as
usual.\newline

There is a natural notion of Hermitian metric on complex
orbifolds. A {\em Hermitian metric g} on a complex orbifold
$(M,J)$ is a $J$-Hermitian metric in the usual sense on the
non-singular part of $(M,J)$ and $G$-invariant in any chart $U/G$.
In such a case, for any chart $U/G$, we have $G\subset U(n)$.
\begin{definition}
A Hermitian metric $g$ on a complex orbifold $(M,J)$ is said to be
{\em strong KT} if the fundamental $2$-form $F$ of $g$ satisfies
$$
\partial\overline{\partial}\,F=0\,.
$$
\end{definition}
\noindent We recall that in general a resolution $(\tilde M, \pi)$
of a singular complex variety $M$ is a normal, nonsingular complex
variety $\tilde M$ with a proper surjective birational morphism
$\pi: \tilde M \to M$. We are interested in particular to resolve
singularities of a complex orbifold endowed with a strong KT
metric in order to obtain a smooth complex manifold admitting a
strong KT metric.
\begin{definition}
Let $(M,J,g)$ be a complex orbifold endowed with a strong KT
metric $g$. A {\em strong KT resolution} of $(M,J,g)$ is the datum
of a smooth complex manifold $(\tilde{M},\tilde{J})$ endowed with
a $\tilde J$-Hermitian strong KT metric $\tilde{g}$ and of a map
$\pi :\tilde{M}\to M$, such that
\begin{enumerate}
\item[i)]$\pi :\tilde{M}\setminus E\to M\setminus S$ is a
biholomorphism, where $S$ is the singular set of $M$ and
$E=\pi^{-1}(S)$ is
the {\em exceptional set};\\
\item[ii)] $\tilde{g} =\pi^*g$ on the complement of a neighborhood
of $E$.
\end{enumerate}
\end{definition}
\noindent In view of the Hironaka Resolution of Singularities
Theorem \cite{Hiro}, the singularities of any complex variety can
be resolved by a finite number of blow-ups. Indeed, if $M$ is a
complex algebraic variety, then there exists a resolution $\pi:
\tilde M \to M$, which is the result of a finite sequence of
blow-ups of $M$. This means that there are varietes $M = M_0, M_1,
\ldots, M_k = \tilde M$, such that $M_j$ is a blow-up of $M_{j -
1}$ along some subvariety with projection $\pi_j: M_j \to M_{j -
1}$ and the map $\pi: \tilde M \to M$ is given by the composition
$\pi = \pi_1 \circ \ldots \circ \pi_k$.

Therefore applying Hironaka's theorem and the results about
blow-ups obtained in Section \ref{blowup} we can prove the
following
\begin{theorem} \label{resolution}
Let $(M,J)$ be a complex orbifold of complex dimension $n$ endowed
with a $J$-Hermitian strong KT metric $g$. Then there exists a
strong KT resolution.
\end{theorem}
\begin{proof} Let $p\in S$ be a singular point of $M$. Take a chart $U_p=\B^n(r)/G_p$, where $\B^n(r)\subset\C^n$ is the ball of
radius $r$ in $\C^n$. Then $X=\C^n/G_p$ is an affine algebraic
variety which has the origin as the only singular point. By
Hironaka (see \cite{Hiro}), there exists a resolution
$\pi_X:\tilde{X}\to X$ which is a quasi-projective variety and it
is obtained by a finite sequence of blow-ups. The set
$E=\pi_X^{-1}(0)$ is a complex submanifold of $\tilde{X}$. Set
$\tilde{U}=\pi^{-1}_X(U_p)$. By identifying $\tilde{U}\setminus E$
with $U_p\setminus \{p\}$, define
$$
\tilde{M}=\left(M\setminus\{p\}\right)\cup\tilde{U}\,.
$$
Now, we define a strong KT metric on $\tilde{M}$ that coincides
with $\pi^*g$ on the complement of a neighborhood of the
exceptional set $E$. \newline Let $\rho :\C^n\to\R$ be the
function defined by $\rho(z)=\sum_{j=1}^n z_j\overline{z}_j$,
$\omega_0=i\,\partial\overline{\partial}\,\rho$ be the standard
K\"ahler form in $\C^n$ and $\iota :\B^n(r)\hookrightarrow \C^n$.
Let $h$ be a non-negative real valued ${\mathcal
C}^\infty$-function such that
\begin{eqnarray*}
h \equiv 1 &{}& \hbox{\rm on}\,\, \B^n\left(\frac13r\right)/G_p\\[3pt]
h \equiv 0 &{}& \hbox{\rm on}\,\, \left(\B^n(r)\setminus
\B^n\left(\frac{2}{3}r\right)\right)/G_p\,.
\end{eqnarray*}
Let $\epsilon\in\R$ and
$$
\tilde{F} =\pi^*_X F +
\epsilon\,i\,\partial\overline{\partial}\,(h \, \iota^*\rho)\,.
$$
Then, $\tilde{F}$ is a $(1,1)$-form on $\tilde{M}$, it is positive
if $\epsilon$ is small enough and satisfies
$\partial\overline{\partial}\,\tilde{F}=0$. It is clear that
$\tilde{F}=\pi^*_XF$ on the complement of a neighborhood of $E$.
The theorem is thus proved.
\end{proof}
We will obtain some applications of Theorem \ref{resolution} in
the next sections. We will show that if even we start with a
nilmanifold with a big first Betti number we can get a new strong
KT manifold with a considerably smaller first Betti number.
\section{A simply-connected example}
We are going to construct a simply-connected strong KT resolution
for the quotient of the torus $\T^6$ by a suitable action of a
finite group.\newline Let $\T^6=\R^{6}/\Z^{6}$ be the standard
torus and denote by $(x_1,\ldots ,x_{6})$ global coordinates on
$\R^{6}$. Define
\begin{equation}\label{torusholomorphic}
\left\{
\begin{array}{lll}
\varphi^1 & = & dx_1 +i\left(f(x)dx_{3}+dx_{4}\right)\\[5pt]
\varphi^2 & = & dx_2 +i\,dx_{5}\\[5pt]
\varphi^3 & = & dx_3 +i\,dx_{6}\,,
\end{array}
\right.
\end{equation}
where $f:\R^{6}\to \R$ is a ${\mathcal C}^\infty$-function. By the
above expression, we easily get
\begin{equation}\label{complexdifferentialtorus}
\left\{
\begin{array}{lll}
d\varphi^1 &=&\,\frac{i}{2}\,df\wedge
(\varphi^3+\overline{\varphi}^3)\,,\\[5pt]
d\varphi^j &=& 0\,,\quad j=2,3\,.
\end{array}
\right.
\end{equation}
Taking in particular $f=f(x_{3},x_{6})$ we have
$$
df = \frac{\partial f}{\partial x_3} dx_3 +\frac{\partial
f}{\partial x_{6}} dx_{6} =\frac{1}{2}\frac{\partial f}{\partial
x_3} (\varphi^3+\overline{\varphi}^3) +\frac{1}{2i}\frac{\partial
f}{\partial x_{6}}(\varphi^3-\overline{\varphi}^3)\,.
$$
Therefore, if $f=f(x_3,x_{6})$ is $\Z^{6}$-periodic, then
\eqref{torusholomorphic} defines a complex structure $J$ on the
torus $\T^{6}=\R^{6}/\Z^{6}$.\newline Let $\sigma :\T^6\to \T^6$
be the involution induced by
$$
(x_1,\ldots ,x_6)\mapsto(-x_1,\ldots ,-x_6)\,.
$$
By choosing $f=f(x_3,x_{6})$ $\Z^{6}$-periodic and even, it
follows that $\sigma$ is $J$-holomorphic.
\newline
The set of singular points for the action of $\sigma$ on $\T^6$ is
given by
$$
S=\left\{ x+\Z^6\,\,\,\vert\,\,\, x\in\frac{1}{2}\Z^6\right\}
$$
and consequently $(M=\T^6/\langle\sigma\rangle,J)$ is a complex
orbifold. Note that $S$ is a set of $64$ points. \newline We have
$$
\sigma^*(\varphi^j)=-\varphi^j\,,\quad j=1,2,3\,.
$$
Denote by
$$
g
=\frac{1}{2}\sum_{j=1}^4\left(\varphi^j\otimes\overline{\varphi}^j+\overline{\varphi}^j\otimes\varphi^j\right)
$$
the natural $\sigma$-invariant Hermitian metric on $\T^6$ and by
$$
F=\frac{i}{2}\sum_{j=1}^3\varphi^j\wedge\overline{\varphi}^j
$$
the corresponding fundamental $2$-form.
\begin{prop}
$(\T^6/\langle\sigma\rangle, J, F)$ is a strong (non-K\"ahler) KT
orbifold.
\end{prop}
\begin{proof} We need only to check that $\partial\overline{\partial}\,F=0$. By a direct computation, taking into
account \eqref{complexdifferentialtorus}, we get
\begin{eqnarray*}
\partial\overline{\partial}\,F &=& \frac{i}{4}\,\partial \left(\frac{\partial f}{\partial x_6}
\varphi^3\wedge\overline{\varphi}^3\wedge\overline{\varphi}^1\right)\\[3pt]
&=&\frac{i}{8}\left[\left(\frac{\partial^2 f}{\partial x_3\partial
x_6}-i\frac{\partial^2 f}{\partial x_6^2}\right)\varphi^3
\right]\wedge
\varphi^3\wedge\overline{\varphi}^3\wedge\overline{\varphi}^1\\[3pt]
&=& 0\,.
\end{eqnarray*}
\end{proof}
According to Theorem \ref{resolution} now we may resolve the
singularities of $\T^6/\langle\sigma\rangle$ in order to obtain a
simply-connected strong KT manifold $\tilde{M}$. More precisely,
for any singular point $p\in S$, we take the blow-up at $p$.

Applying the same argument used by Joyce as in \cite[Lemma 12.1.1
and 12.1.2]{J}, we get that the orbifold fundamental group
$\pi_1\left(\T^6/\langle\sigma\rangle \right)$ is abelian. Hence,
since
$$
b_1(\T^6/\langle\sigma\rangle)=\dim_\R\left\{\alpha\in\Lambda^1(\T^6)\,\,\,\vert\,\,\,\alpha
\,\, \hbox{\rm is harmonic and}\,\,\sigma\hbox{\rm
-invariant}\right\}=0\,,
$$
we deduce that the strong KT resolution $\tilde{M}$ of the
orbifold $\T^6/\langle\sigma\rangle$ is simply-connected.
Moreover, in a similar way we have
$$
b_{2 j + 1} (\T^6/\langle\sigma\rangle) =0 \quad {\mbox{and}}
\quad b_{2 j + 2} (\T^6/\langle\sigma\rangle) =b_{2 j + 2} (\T^6),
$$
for any $j =0, 1,2$.

\begin{rem}{\rm
The same construction can be generalized to the $2n$-dimensional
orbifold $\T^{2n}/\langle \sigma\rangle$, where the involution
$\sigma :\T^{2n}\to \T^{2n}$ is defined by
$$
\sigma\left((x_1,\ldots ,x_{2n})\right)=(-x_1,\ldots ,-x_{2n})
$$
and the complex structure on $\T^{2n}/\langle \sigma\rangle$ is
given by
\begin{equation}
\left\{
\begin{array}{lll}
\varphi^1 & = & dx_1 +i\left(f(x)dx_{n}+dx_{n+1}\right)\,,\\[10pt]
\varphi^2 & = & dx_2 +i\,dx_{n+2}\,,\\[10pt]
&\vdots&\\[10pt]
\varphi^n & = & dx_n +i\,dx_{2n}\,,
\end{array}
\right.
\end{equation}
where $f:\R^{2n}\to \R$ is an even, $\Z^{2n}$-periodic ${\mathcal
C}^\infty$-function of the two variables $(x_n,x_{2n})$. }
\end{rem}
\section{ $8$-dimensional examples}
We are going to construct two new $8$-dimensional compact examples
with first Betti number equal to one starting from the
Kodaira-Thurston surface \cite{Kod}, which is the only nilmanifold
of real dimension 4 admitting a strong KT structure besides the
$4$-dimensional torus ${\mathbb T}^4$. We will consider an action
of a finite group on the product of two Kodaira-Thurston surfaces
or on a product of a Kodaira-Thurston surface for $\T^4$
preserving the standard strong KT structures on the products. One
of the two previous actions is similar to the one considered in
\cite{FM} in the context of symplectic geometry.

\subsection{Product of a Kodaira-Thurston surfaces and a torus}
Consider on $\C^4$ the following product
\begin{eqnarray*}
(z_1, z_2, z_3, z_4) \star (w_1, w_2, w_3, w_4)& = & (z_1 + w_1, z_2 + w_2 + \frac 14 i (z_1 \overline w_1 - \overline z_1 w_1), \\
&{}& z_3 + w_3, z_4 + w_4 ),
\end{eqnarray*}
for any $z_j, w_j \in \C$, $j=1, 2, 3, 4$.

The corresponding real nilpotent Lie group $N$ is the product $H_3
\times \R^5$, where $H_3$ is the real $3$-dimensional Heinsenberg
group and it has a left-invariant complex $J$ defined by the
$(1,0)$-forms
$$
\varphi^1 = d z_1, \quad \varphi^2 = dz_2 - \frac 14 i (z_1
d\overline {z}_1 - \overline z_1dz_1), \quad \varphi^3 = dz_3,
\quad\varphi^4=dz_4 \,.
$$

Let $\Lambda$ be the lattice generated by $1$ and $\xi = e^{ \frac
{2 \pi i} {3}}$ and consider the discrete subgroup $\Gamma \subset
N$ formed by the elements $(z_1, z_2, z_3, z_4)$ such that $z_1,
z_2, z_3, z_4 \in \Lambda$. The compact quotient $(M = \Gamma
\backslash N, J)$ is a complex nilmanifold and it can be viewed as
a principal torus bundle
$$
{\mathbb T}^2 = \C / \Lambda \rightarrow M \rightarrow {\mathbb
T}^6 = \left(\C / \Lambda \right)^3.
$$
$M$ has a natural strong KT metric compatible with $J$ defined by
the $(1,1)$-form
$$
F = \frac{i}{2} \sum_{j = 1}^4 \varphi^j \wedge \overline
\varphi^j,
$$
since
$$
\begin{array}{l}
d \varphi^ j = 0\,, j = 1,3,4,\\[5pt]
d \varphi^2 = - \frac 12 i \, \varphi^1 \wedge \overline
\varphi^1.
\end{array}
$$

Consider the following action of the finite group $\Z_3$
$$
\begin{array}{c}
\lambda: N \rightarrow N, \\[3pt]
(z_1, z_2, z_3, z_4) \mapsto (\xi z_1, z_2, \xi z_3, \xi z_4)\,.
\end{array}
$$
One has that $\lambda (a \star b) = \lambda (a) \star \lambda
(b)$, for any $a, b \in N$. Moreover, $\lambda (\Gamma)= \Gamma$
and therefore there is an induced action $\lambda$ on the quotient
$M$. Since the action on the $(1,0)$-forms is given by
\begin{equation}\label{invariantformskodairatori}
\lambda^* \varphi^1 = \xi \varphi^1, \quad \lambda^* \varphi^2 =
\varphi^2, \quad \lambda^* \varphi^3 = \xi \varphi^3, \quad
\lambda^* \varphi^4 = \xi \varphi^4,
\end{equation}
the $(1,1)$-form $F$ is $\Z_3$-invariant and therefore it induces
a strong KT metric on the complex orbifold $\hat M = M / \Z_3$.

By Theorem \ref{resolution} there exists a smooth compact strong
KT manifold $(\tilde M, \tilde F)$ which is biholomorphic to
$(\hat M, F)$ outside the singular points.

For any singular point $p \in \hat M$, that we can consider as
$(0,0,0,0)$ in our coordinates by translating by an element of the
group $N$, we resolve the singularity of $\B^4(r) / \Z_3$ with a
non-singular K\"ahler model, blowing up $\B^4(r)$ to get $\tilde
B$. In this way we replace the point with a complex projective
space $\CP^3$ in which $\Z_3$ acts as
$$
[z_1, z_2, z_3, z_4] \to [\xi z_1, z_2, \xi z_3, \xi z_3] = [z_1,
\xi^2 z_2, z_3, z_4].
$$
then there are two components of the fix-point locus of the
$\Z_3$-action on $\tilde B$: the point $q = [0,1, 0,0]$ and the
complex projective plane
$$
\Proj \left \{ [z_1, 0, z_3, z_4] \right \} \subset \CP^3\,.
$$
The resolution of $\B^4 (r)/ \Z_3$ is then obtained blowing-up
$\tilde B$ at $q$ and $\Proj \left \{ [z_1, 0, z_3, z_4] \right
\}$ and doing the quotient by the action of $\Z_3$. \newline
Consider the $\Z_3$-invariant Hermitian metric on $M$ defined by
$$
g
=\frac{1}{2}\sum_{j=1}^4\left(\varphi^j\otimes\overline{\varphi}^j+\overline{\varphi}^j\otimes\varphi^j\right)\,.
$$
By \eqref{invariantformskodairatori} we easily obtain that the
only harmonic $\Z_3$-invariant $1$-form on $M$ is $dy_2$, where we
denote $z_2 =x_2+iy_2$. Since $M = \Gamma \backslash N$ is a
compact nilmanifold, we have that the De Rham cohomology is just
given by harmonic left-invariant forms and
$$
b_k(M /
\Z_3)=\dim_\R\left\{\alpha\in\Lambda^k(M)\,\,\,\vert\,\,\,\alpha
\,\,\hbox{\rm is harmonic and}\,\,\Z_3\hbox{\rm
-invariant}\right\}\,.
$$
Therefore, we conclude that
$$b_1(M / \Z_3) =1.
$$
As in \cite{FM} then one has that
\begin{equation}\label{b1}
H^1(\tilde M) = H^1 (M/ \Z_3) \oplus \left(\bigoplus_i H^1 (E_i)
\right)
\end{equation}
where $E_i$ is the exceptional divisor corresponding to each fixed
point $p_i$ and therefore $b_1 (\tilde M) = b_1(M / \Z_3) =1$.
Therefore the topology of the new strong KT manifold is simpler
with respect to the one of $M$ since by using Nomizu's Theorem
\cite{Nom} we had $b_1(M) = 7$ for the nilmanifold $M$.
\subsection{Product of two Kodaira-Thurston surfaces}
Consider on $\C^8$ the following product
$$
\begin{array}{lcl}
(z_1, z_2, z_3, z_4) \star(w_1, w_2, w_3, w_4) &=&\left( z_1 + w_1, z_2 + w_2 + \frac 14 i (z_1 \overline w_1 - \overline z_1 w_1),\right.\\[3pt]
&{}& \left.z_3 + w_3, z_4 + w_4 + \frac 14 i (z_3\overline w_3 -
\overline z_3 w_3) \right),
\end{array}
$$
for any $z_j, w_j \in \C$, $j=1, 2, 3, 4$.

The corresponding real nilpotent Lie group $N$ is the product $H_3
\times H_3 \times \R^2$ and it has a left-invariant complex $J$
defined by the $(1,0)$-forms
$$
\begin{array}{l}
\varphi^1 = d z_1, \quad \varphi^2 = dz_2 - \frac 14 i (z_1
d\overline
{z}_1 - \overline z_1dz_1), \\[4pt]
 \varphi^3 = dz_3, \quad\varphi^4=dz_4 - \frac 14 i (z_3 d\overline
{z}_3 - \overline z_3dz_3)\,.
\end{array}
$$

Let $\Lambda$ be the lattice generated by $1$ and $\xi = e^{ \frac
{2 \pi i} {3}}$ and consider the discrete subgroup $\Gamma\subset
N$ formed by the elements $(z_1, z_2, z_3, z_4)$ such that $z_1,
z_2, z_3, z_4 \in \Lambda$. The compact quotient $(M = \Gamma
\backslash N, J)$ is a complex nilmanifold and it can be viewed as
a principal torus bundle
$$
{\mathbb T}^2 = \C / \Lambda \rightarrow M \rightarrow {\mathbb
T}^6 = \left(\C / \Lambda \right)^3.
$$
$M$ has a natural strong KT metric compatible with $J$ defined by
the $(1,1)$-form
$$
F = \frac{i}{2} \sum_{j = 1}^4 \varphi^j \wedge \overline
\varphi^j,
$$
since
$$
\begin{array}{l}
d \varphi^ j = 0\,,\quad j = 1,3\,,\\[5pt]
d \varphi^2 = - \frac 12 i \, \varphi^1 \wedge \overline \varphi^1,\\[5pt]
d \varphi^4 = - \frac 12 i \, \varphi^3 \wedge \overline
\varphi^3.
\end{array}
$$

Consider the following action of a finite group $G$
$$
\begin{array}{c}
\lambda: N \rightarrow N, \\[3pt]
(z_1, z_2, z_3, z_4) \mapsto (\xi z_3, z_4, \xi z_1, z_2)\,.
\end{array}
$$
One has that $\lambda (a \star b) = \lambda(a) \star \lambda (b)$,
for any $a, b \in N$. Moreover, $\lambda (\Gamma)= \Gamma$ and
therefore there is an induced action $\lambda$ on the quotient
$M$. Since the action on the $(1,0)$-forms is given by
\begin{equation}\label{invariantformskodairakodaira}
\lambda^* \varphi^1 = \xi \varphi^3, \quad \lambda^* \varphi^2 =
\varphi^4, \quad \lambda^* \varphi^3 = \xi \varphi^1, \quad
\lambda^* \varphi^4 = \varphi^2,
\end{equation}
the $(1,1)$-form $F$ is $G$-invariant under the previous action
and therefore it induces a strong KT metric on the complex
orbifold $\hat M = M /G$.

By Theorem \ref{resolution} there exists a compact strong KT
resolution $(\tilde M, \tilde F)$ which is biholomorphic to $(\hat
M, F)$ outside the singular set.

By using \eqref{b1} and the same argument as for the previous
example we get
$$
b_1(M /G) =1 = b_1 (\tilde M)\,
$$
while for the manifold $M$ we had $b_1(M) = 6$ by \cite{Nom}.

\end{document}